\documentclass[10pt,a4paper]{article}
\usepackage{hyperref}
\usepackage{amsmath,enumerate, amsthm}
\usepackage{amssymb}
\usepackage{tikz}
\usepackage[hmargin=3cm,vmargin=3cm]{geometry}

\newtheorem{thm}{Theorem}

\newtheorem{prop}[thm]{Proposition}
\newtheorem{lemma}[thm]{Lemma}
\newtheorem{obs}[thm]{Observation}

\def\X{X}
\def\Y{Y}
\def\Z{Z}

\usetikzlibrary{patterns}

\definecolor{roug}{RGB}{0,0,0} 
\tikzstyle mygrid=[line width=0.5,color=black!15]
\tikzstyle mygrid2=[line width=1,color=black!20]
\tikzstyle{gridnode}=[shape=circle,fill=black,draw=black,minimum size=0.5pt,inner sep=0.5pt]
\tikzstyle{patternus}=[draw=roug,thick,rectangle,inner sep=7pt,fill=roug!15, semitransparent]
\tikzstyle{patternusbenef}=[draw=roug,fill=roug!05,semitransparent]
\tikzstyle{labelnode}=[circle,fill=white, inner sep =2,minimum size=0 pt]
\tikzstyle{labelnode2}=[circle,fill=white, inner sep =2,minimum size=0 pt,anchor=mid]
\tikzstyle{code}=[shape=circle,fill=black,draw=black,minimum size=0.5pt,inner sep=2pt]

\def\framegrid{
\begin{scope}
\clip (-2.9,-2.9) rectangle (1.9,1.9);
\draw[mygrid] (-2.9,-2.9) grid (1.9,1.9);
\draw[rotate = 45,mygrid] (-4.9,-4.9) grid [xstep=1.414,ystep =1.414] (3.9,3.9);
\draw[shift = {(1,0)}, rotate = 45, mygrid] (-4.9,-4.9) grid [xstep=1.414,ystep =1.414] (3.9,3.9);
\foreach \I in {-2,...,1}\foreach \J in {-2,...,1}
	{\node[gridnode](\I\J) at (\I,\J) {};}
\end{scope}
}

\begin{document}

\title{An improved lower bound for $(1,\leq 2)$-identifying codes in the king grid\thanks{This research is
    supported by the ANR project IDEA, under contract {ANR-08-EMER-007},
    2009-2011.}}
\author{
Florent Foucaud\thanks{LaBRI, Universit\'e de Bordeaux, 351 cours de la Lib\'eration, 33405
Talence cedex, France, e-mail:foucaud@labri.fr}
\and
Tero Laihonen\thanks{Department of Mathematics, University of Turku,
20014 Turku, Finland, e-mail: terolai@utu.fi}
\and
Aline Parreau\thanks{Institut Fourier (Universit{\'e} Joseph Fourier,
CNRS), St.\ Martin d'H{\`e}res, France, e-mail: aline.parreau@ujf-grenoble.fr}}

\maketitle

\begin{abstract}
 We call a subset $C$ of vertices of a graph $G$ a $(1,\leq \ell)$-identifying code if for all subsets $X$ of vertices  with size at most $\ell$, the sets $\{c\in C |\exists u \in X, d(u,c)\leq 1\}$ are distinct. The concept of identifying codes was introduced in 1998 by Karpovsky, Chakrabarty and Levitin. Identifying codes have been studied in various grids. In particular, it has been shown that there exists a $(1,\leq 2)$-identifying code in the king grid with density 
$\frac{3}{7}$ and that there are no such identifying codes with density smaller than $\frac{5}{12}$. 
Using a suitable frame and a discharging procedure, we improve the lower bound by showing that any $(1,\leq 2)$-identifying code of the king grid has density at least $\frac{47}{111}$.
\end{abstract}

\section{Introduction}

Let $G=(V,E)$ be a simple undirected graph with vertex set $V$ and
edge set $E$. The closed neighbourhood of a vertex $v\in V$, which
consists of the vertex itself and all the adjacent vertices, is
denoted by $N[v]$. In addition, we write $N[X]=\bigcup_{v\in X}
N[v]$ for $X\subseteq V$. In \cite{hon2}, a subset $C\subseteq V$ is
called a $(1,\leq \ell)$-\emph{identifying code} if
$$N[X]\cap C \neq
N[Y]\cap C$$ for any two distinct subsets $X\subseteq V$ and
$Y\subseteq V$ of size at most $\ell$ (one of them can be the empty
set). The elements of a code $C$ are called \emph{codewords}.

Identifying codes were introduced by Karpovsky, Chakrabarty and
Levitin in \cite{kar1} and can be applied to locate objects
in sensor networks \cite{ray}. A network is modelled by a graph and
a sensor can check its closed neighbourhood. It gives an alarm if it
detects at least one of the sought objects there. Suppose we have a
$(1,\le \ell)$-identifying code $C$ in the graph, and we place the
sensors to the vertices corresponding to the codewords of $C$. Then,
knowing the set of alarming sensors $A\subseteq C$, we can determine
where the objects are (assuming that there are at most $\ell$ of
them). Indeed, we have $A=N[X]\cap C$ for some subset $X$ of size at
most $\ell$, and because these are unique, we can determine $X$
--- the set of vertices where the objects are.

Of course, we would like to use as few sensors as possible, so our
aim is to find identifying codes with smallest possible cardinality.
For infinite grids, we define below the measure \emph{density} for
this purpose.

Identifying codes have been considered, for example, in the
following infinite grids: the square grid, the triangular grid, the
king grid and the hexagonal mesh \cite{lits,char1,coh0,coh3,king,
cranston,hon3,hon4,hon5}. For more papers concerning the topic of
identification, see \cite{lobwww}. In this paper, we focus on
$(1,\le 2)$-identifying codes in the king grid. The king grid has
vertex set $\mathbb{Z}^2$ and two vertices are adjacent if and only if
the Euclidean distance between them is at most $\sqrt{2}.$ Hence,
the closed neighbourhood of a vertex consists of nine vertices (see Figure \ref{fig:king}).

\begin{figure}[h]
\begin{center}
\begin{tikzpicture}[scale=0.6]

\clip (-3.7,-2.7) rectangle (6.7,3.7);

\draw[mygrid2] (-6.9,-5.9) grid (6.9,5.9);
\draw[rotate = 45,mygrid2] (-15.9,-15.9) grid [xstep=1.414,ystep =1.414] (15.9,15.9);
\draw[shift = {(1,0)}, rotate = 45, mygrid2] (-10.9,-10.9) grid [xstep=1.414,ystep =1.414] (10.9,10.9);

\foreach \I in {-6,...,6}\foreach \J in {-5,...,5}
	{\node[gridnode](\I\J) at (\I,\J) {};}

\foreach \J in {-5,...,5}
\foreach \I in {-13,-11,-9,-6,-4,-2,1,3,5}
{\node[code] at (\I+\J+5,\J) {};}

\draw (-2.5,-1.5) rectangle (4.5,-0.5);
\draw[dashed] (-1.5,-0.5) rectangle (5.5,0.5);
\end{tikzpicture}

\end{center}
\caption{\label{fig:king} The infinite king grid with a $(1,\le 2)$-identifying code of density $\frac{3}{7}$.}
\end{figure}
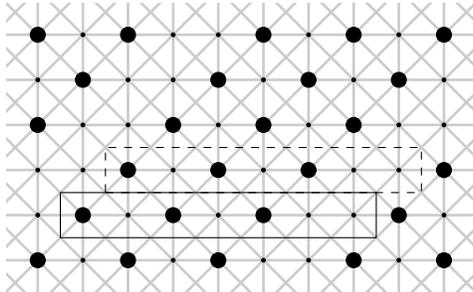

In \cite{hon2}, a construction of a $(1,\le 2)$-identifying code in
the infinite king grid with density $\tfrac{3}{7}$ is given (see Figure \ref{fig:king}). Recently, in
\cite{mikko}, a lower bound of $\tfrac{5}{12}$ was obtained. In this paper,
we show that our approach which utilizes certain {\em frames}, improves
this lower bound to $\frac{47}{111}$.

\section{Preliminaries}

We call {\em code} any subset of $\mathbb{Z}^2$.
As usual, the {\em density} $D(C)$ of a code $C\subseteq \mathbb{Z}^2$  is defined by:
$$D(C)=\limsup_{n\to \infty} \frac{|C\cap Q_n|}{|Q_n|}$$
where $Q_n=\{(x,y)\in \mathbb{Z}^2\mid |x|\le n, |y|\le n\}$.

The next theorem gives the previously known bounds for the optimal density of a $(1,\leq 2)$-identifying code in the king grid.

\begin{thm}[\cite{king,mikko}]\label{thm:previous}
Let $d$ be the optimal density of a $(1,\leq 2)$-identifying code in the king grid. Then $0.41\overline{6}=\tfrac{5}{12}\leq d\leq \tfrac{3}{7}=0.4286...$.
\end{thm}

To prove bounds on the density of a code $C$, one method is to use a finite subset
$X$ of $\mathbb{Z}^2$. For any vertex $v\in \mathbb{Z}^2$, we denote by $v+X$ or $X+v$ the set $\{v+u\ |\ u\in X\}$, and by $n(X,C)$ the average number of codewords of $C$ in the translations of $X$: $n(X,C)=\limsup_{n\to \infty} \frac{\sum_{v\in Q_n} |(v+X)\cap C|}{|Q_n|}$. If we have some knowledge about $n(X,C)$, we can get results on the density of $C$ with the following proposition.

\begin{prop}\label{prop:dens}
Let $C$ be a code of $\mathbb{Z}^2$ and $X$ a nonempty finite subset of $\mathbb{Z}^2$, then:
$$D(C)=\frac{n(X,C)}{|X|}.$$
\end{prop}

\begin{proof}
Let $h$ be an integer such that $X\subseteq Q_h$.
Let $n$ be a positive integer. 
A vertex of $(v+X)\cap C$, for $v\in Q_{n+h}$ is either in $C\cap Q_n$, or in $Q_{n+2h}\setminus Q_n$. Among all the sets $(v+X)\cap C$, with $v\in Q_{n+h}$, each vertex of $C\cap Q_n$ is counted exactly $|X|$ times and each vertex of  $Q_{n+2h}\setminus Q_n$ is counted at most $|X|$ times. Hence we have:

$$|X||C\cap Q_n|\leq \sum_{v\in Q_{n+h}} |(v+X)\cap C|\leq |X|(|C\cap Q_n| + |Q_{n+2h}\setminus Q_n|).$$

By dividing every term by $|Q_n|$, we obtain:
$$|X|\cdot\frac{|C\cap Q_n|}{|Q_n|}\leq \frac{\sum_{v\in Q_{n+h}} |(v+X)\cap C|}{|Q_n|} \leq |X|\cdot\frac{|C\cap Q_n|}{|Q_n|}+|X|\cdot\frac{|Q_{n+2h}\setminus Q_n|}{|Q_n|}.$$

Since $h$ is a fixed integer, we have: $\limsup\limits_{n\to \infty} \frac{|Q_{n+2h}\setminus Q_n|}{|Q_n|} =0$. As $n\to \infty$ we obtain:

$$\limsup_{n\to \infty} \frac{\sum_{v\in Q_{n+h}} |(v+X)\cap C|}{|Q_n|}= |X|\cdot D(C).$$

To end the proof, one can easily verify that $\limsup\limits_{n\to \infty} \frac{\sum_{v\in Q_{n+h}} |(v+X)\cap C|}{|Q_n|}= n(X,C)$. 


\end{proof}

The identifying code problem can be seen as a covering problem. Indeed, a code $C$ is a $(1,\leq \ell)$-identifying code if and only if, for any two distinct subsets of vertices, $X$ and $Y$, of size at most $\ell$, the symmetric difference $N[X]\Delta N[Y]$ is covered by at least one element of $C$. 
The following theorem, from \cite{king}, reduces (in the case of the king grid) the family of sets that a code has to cover to be a $(1,\leq 2)$- identifying code.

\begin{thm}[\cite{king}, Theorem 2.1]\label{thm:equiv}
A code $C$ is a $(1,\leq 2)$-identifying code in the king grid if
and only if, for all $(x,y)\in \mathbb{Z}^2$, the sets
\begin{enumerate}
\item\label{set4} $\{(x,y), (x+3,y), (x,y+3), (x+3,y+3)\}$,
\item\label{sethori} $\{(x,y), (x+1,y), (x+2,y)\}$
\item\label{setvert} $\{(x,y),(x,y+1),(x,y+2)\}$
\end{enumerate}
each contain at least one codeword of $C$.
\end{thm}

Note that Theorem~\ref{thm:equiv} together with Proposition~\ref{prop:dens} directly implies that the density of a $(1,\leq 2)$-identifying code of the king grid is at least $\frac{1}{3}$.
To improve this lower bound (and the one of Theorem~\ref{thm:previous}), we will use, as the set $X$ of Proposition~\ref{prop:dens}, the set of Figure \ref{fig:not} (without loss of generality, one can set the left-lower corner of this set to be the origin). We call {\em frame} any set isomorphic to the set of Figure~\ref{fig:not}. The aim of this paper is to give a lower bound on the average number of codewords in a frame. More precisely, we will use a discharging procedure to show that the average number of codewords in a frame is at least $5+\frac{3}{37}$. 
Hence, we will need to consider the {\em neighbourhood} of a frame. There is a natural bijection between all the frames and the set $\mathbb{Z}^2$. We will also consider the {\em frames lattice}  over the set of all frames, where the distance between two frames will be the distance in the king grid between the two corresponding vertices. As an example, the {\em $2$-ball} of a frame $F$ is the set of frames $$\{F+(x,y) \ | \ (x,y)\in \mathbb{Z}^2, \max\{|x|,|y|\} \leq 2\}.$$

We will often study a frame together with all the vertices at distance at most two of some vertex of the frame. Therefore, we will use the notation {\bf Xy} with {\bf X}$\in${\bf A-H} and {\bf y}$\in${\bf a-h}, for a vertex on line {\bf X} and row {\bf y} with respect to the coordinates of Figure \ref{fig:not}.
The four vertices in positions  {\bf Cc},  {\bf Cf},  {\bf Fc}  and {\bf Ff} are called the {\em corners} of the frame. A {\em side} of a frame is a set of four vertices of the frame lying in the same line or column. 

The conditions of Theorem~\ref{thm:equiv} can be reformulated using our terminology.
A code $C$ is a $(1,\leq 2)$-identifying code in the king grid, if
and only if for each frame $F$:

\vspace{0.4cm}

\noindent{\bf Condition 1:} At least one corner of $F$ is a codeword of $C$.

\vspace{0.2cm}

\noindent{\bf Condition 2:} Each set of three consecutive vertices on a side of $F$contains at least one codeword of $C$.

\vspace{0.4cm}

\begin{figure}[ht]
\begin{center}
\begin{tikzpicture}[scale=0.5]
\begin{scope}
\clip (-4.7,-4.7) rectangle (3.7,3.7);
\draw[mygrid] (-4.9,-4.9) grid (3.9,3.9);
\draw[rotate = 45,mygrid] (-15.9,-15.9) grid [xstep=1.414,ystep =1.414] (15.9,15.9);
\draw[shift = {(1,0)}, rotate = 45, mygrid] (-10.9,-10.9) grid [xstep=1.414,ystep =1.414] (10.9,10.9);
\foreach \I in {-4,...,3}\foreach \J in {-4,...,3}
	{\node[gridnode](\I\J) at (\I,\J) {};}
\end{scope}

\foreach \I / \A in {3/A, 2/B, 1/C, 0/D, -1/E, -2/F, -3/G, -4/H}
	{\node[labelnode] at (-5,\I) {\A};}
\foreach \I / \A in {3/a, 2/b, 1/c, 0/d, -1/e, -2/f, -3/g, -4/h}
	{\node[labelnode2] at (-1-\I,4) {\A};}
	
\begin{scope}[even odd rule]
\filldraw[patternus] (-2.5,-2.5) rectangle (1.5,1.5)
 (-1.5,-1.5) rectangle (0.5,0.5);
\end{scope}
\end{tikzpicture}
\end{center}
\caption{\label{fig:not} Notation for vertices in the neighbourhood of a frame.}
\end{figure}
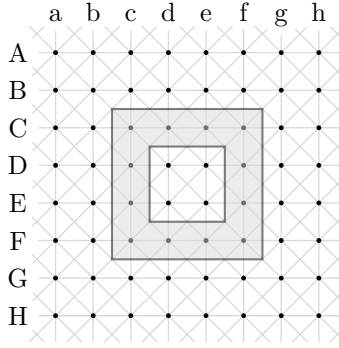

In the following, we assume that $C$ is a $(1,\leq 2)$-identifying code of the king grid.
Given an integer $k$, a {\em $k$-frame} (resp. {\em $k^+$-frame}) is a frame containing exactly $k$ (resp. at least $k$) vertices of $C$.

\begin{lemma}\label{lem:fourcorners}
Let  $F$ be a frame, then $F$ is a  $4^+$-frame. If $F$ is a $4$-frame, then all its codewords are in the corners of $F$.
\end{lemma}

\begin{proof}
It is easily observed that any frame $F$ can be partitioned into four sets of three consecutive vertices on a line or column. Hence, by Condition 2, $F$ contains at least four codewords.

Assume that $F$ contains exactly four codewords. By Condition 1, one of the codewords must be in a corner, say in position {\bf Cc}. Then, by Condition 2, there must be another codeword on line {\bf C}. If it was on column {\bf d} or {\bf e}, one could partition the remaining vertices into three sets of three consecutive vertices on a line or column and $F$ would be a $5^+$-frame, a contradiction. So there is a codeword in position {\bf Cf}. Using similar arguments, the two other codewords must be in positions {\bf Fc} and {\bf Ff}.
\end{proof}

\begin{obs}\label{obs:NC}
If a frame has three non-corner codewords on two of its sides, then it is a $6^+$-frame.
\end{obs}
\begin{proof}
Assume first that the three non-corner codewords lie on opposite sides. Without loss of generality, we can assume there are in positions {\bf Cd}, {\bf Ce} and {\bf Fd}.
By Condition 1, there must be one codeword in a corner. Assume it is in column  {\bf c}. Then, by Condition 2,  there must be another codeword in column {\bf c} and one in column {\bf f}, leading to a $6^+$-frame. If the corner codeword is in column {\bf f}, then there would be two codewords in column {\bf f} and one in column {\bf c}, leading again to a $6^+$-frame.

We assume now that the three non-corner codewords are on adjacent sides. Without loss of generality, we can assume they are not in the set $S$ composed by line  {\bf C} and column  {\bf f}. By Condition 2, there are at least two codewords in $S$. By Condition 1, there is a codeword in a corner. If there is a codeword in position {\bf Fc} then $F$ is a $6^+$-frame. Otherwise, there is a corner codeword in $S$ and then there must be three codewords in $S$, leading to a $6^+$-frame.
\end{proof}

\begin{obs}\label{obs:CO}
If a frame $F$ has a corner codeword $c$ and two non-corner codewords on one of the sides of $F$ that $c$ does not belong to, then $F$ is a $6^+$-frame.
\end{obs}
\begin{proof}
Without loss of generality, we assume that there are three codewords in positions {\bf Cc}, {\bf Fd} and {\bf Fe}. By Condition 2, there must be another codeword in column {\bf c}, another codeword in line {\bf C} and one among positions {\bf Df}, {\bf Ef} and {\bf Ff}. Hence, $F$ is a $6^+$-frame.
\end{proof}

\begin{lemma}\label{lem:four}
Let $F$ be a $4$-frame. Then the four frames $F+\{(0,2), (0,-2), (2,0), (-2,0)\}$ are $6^+$-frames.
\end{lemma}

\begin{proof}
By Lemma~\ref{lem:fourcorners}, the four codewords of $F$ are its four corners.
By Condition $1$ applied on $F+(1,0)$, there is a codeword on either position {\bf Cg} or position {\bf Fg}. 
Hence by Observation~\ref{obs:NC}, $F+(2,0)$ is a $6^+$-frame.
The claim is obtained by symmetry.
\end{proof}

Lemma~\ref{lem:four} gives a short proof of a result first given in \cite{mikko}. The proof is a warm-up for our main result.

\begin{thm}[\cite{mikko}]\label{thm:mikko}
The density of any $(1,\leq 2)$-identifying code of the king grind is at least $\tfrac{5}{12}$.
\end{thm}

\begin{proof}
Let $C$ be a $(1,\leq 2)$-identifying code of the king grid. We use the discharging method to show that the average number of codewords in a frame is at least $5$. Since the size of each frame is $12$, the result will follow by Proposition~\ref{prop:dens}.

In the beginning, each $k$-frame has charge $k$. Then, each $6^+$-frame $F$ gives charge $\tfrac{1}{4}$ to each $4$-frame among frames $F+\{(0,2), (0,-2), (2,0), (-2,0)\}$. By Lemma~\ref{lem:fourcorners}, there exist only $4^+$-frames. By Lemma~\ref{lem:four}, each $4$-frame receives charge $1$ and each $6^+$-frame gives at most charge $1$ away. So after the discharging process, each frame has at least charge $5$. Each frame only gives charge to vertices at distance at most 2, hence the average number of codewords of $C$ in a frame is at least $5$ and we are done.
\end{proof}

In the following, we will prove our main result:

\begin{thm}\label{thm:main}
The density of any $(1,\leq 2)$-identifying code of the king grid is at least $\tfrac{47}{111}=0.\overline{423}$.
\end{thm}

To prove this theorem, we will use a similar technique than the one of the proof of Theorem \ref{thm:mikko}, but giving a final charge of more than $5$ to each frame. Hence, a charge of $5$ can be seen as the reference value for the charge of a frame. We say that the {\em charge excess} of a $k$-frame is $k-5$. The {\em charge excess} within a subset $S$ of frames is the sum of the charge excesses of all the frames of $S$.

\section{Structural properties of code $C$}

We now prove some results on the structure of $C$ in the viewpoint of the frames.
We call {\em $4$-benefactor} a $6^+$-frame $F$ having a $4$-frame among frames $F+\{(0,2), (0,-2), (2,0), (-2,0)\}$. 

\begin{lemma}\label{lem:4benef}
Let $F$ be a $4$-benefactor. Then $F$ has $6^+$-frames in each corner position of its $2$-ball being at distance~$2$ from the $4$-frames among $F+\{(0,2), (0,-2), (2,0), (-2,0)\}$.
Moreover, if $F$ is a $6$-frame, there is a unique $4$-frame among frames $F+\{(0,2), (0,-2), (2,0), (-2,0)\}$.
If $F$ is a $7^+$-frame, there are at most two $4$-frames among frames $F+\{(0,2), (0,-2), (2,0), (-2,0)\}$.
\end{lemma}

\begin{proof}
The first part of the claim is a direct consequence of Lemma~\ref{lem:four}.

There is at least one $4$-frame among $F+\{(0,2), (0,-2), (2,0), (-2,0)\}$, without loss of generality, we can assume that $F+(-2,0)$ is a $4$-frame.
Then $F+(2,0)$ cannot be a $4$-frame since then, no corner of $F$ is a codeword, contradicting Condition $1$ on $F$.
Suppose now that there is another $4$-frame among $F+\{(0,2), (0,-2)\}$, say $F+(0,2)$. As before, $F+(0,-2)$ cannot be a $4$-frame.
By Condition $1$ applied on $F$, there is a codeword in position {\bf Ff}. By Condition 1 applied on $F+(-1,0)$ and $F+(0,1)$, there are two codewords among positions {\bf Ce, Fe, Ec} and {\bf Ef}.
Hence $F$ is a $7^+$-frame and we are done.
\end{proof}

\begin{figure}[h!]
\begin{center}
\begin{tikzpicture}[scale=0.8]
\draw[mygrid] (-2,-2) grid (3,3);
\draw[line width=2pt] (0,0) rectangle (1,1);
\draw[line width=2pt] (-2,-2) rectangle (3,3);
\node at (0.5,0.5) {\Large $\bf F$};
\node at (-1.5,-1.5) {$6^+$};
\node at (-1.5,2.5) {$6^+$};
\node[labelnode] at (-1.5,0.5) {$4$};
\foreach \I in {-0.5,0.5,1.5,2.5}
	{\node[labelnode] at (\I,2.5) {$\X$};
	\node[labelnode] at (\I,-1.5) {$\X$};}
\foreach \I in {-0.5,0.5,1.5}
{\node[labelnode] at (2.5,\I) {$\X$};
\node[labelnode] at (1.5,\I) {$\Z$};}

\node[labelnode] at (0.5,1.5) {$\Z$};
\node[labelnode] at (0.5,-0.5) {$\Z$};
\node[labelnode] at (-0.5,-0.5) {$\Y$};
\node[labelnode] at (-0.5,1.5) {$\Y$};
\node[labelnode] at (-1.5,-0.5) {$\X$};
\node[labelnode] at (-1.5,1.5) {$\X$};
\node[labelnode] at (-0.5,0.5) {$\X$};
\end{tikzpicture}
\end{center}
\caption{\label{fig:4benef} $2$-ball around a $4$-benefactor $6$-frame.}
\end{figure}

 The $6^+$-frames in the corner positions of a $4$-benefactor $6^+$-frame described in the previous lemma are called the {\em co-benefactors} of $F$. The next lemma is valid for other cases by smmetry.

\begin{lemma}\label{lem:4benef6}
Let $F$ be a $4$-benefactor $6$-frame oriented as in Figure~\ref{fig:4benef}. Then:
\begin{itemize}
\item Frame $F$ has at least two $6^+$-frames $F_1$, $F_2$ in its $2$-ball (in addition to its co-benefactors).
\item If $F$ has no $6^+$-frame in a $\Y$- or $\Z$-position, then $F$ has at least three $6^+$-frames in its $2$-ball (in addition to its co-benefactors), or one of its co-benefactors is a $7^+$-frame and $F+(-1,0)$ is not a $6^+$-frame.
\end{itemize}
\end{lemma}

\begin{proof}
In this proof, we will often use Conditions 1 and 2 without explicitly referring to them.

By Condition 1 applied on $F+(-1,0)$, there is a codeword in either {\bf Ce} or {\bf Fe}. Without loss of generality, we may assume that there is a codeword in position {\bf Ce} (the other case follows by symmetry).

Since $F$ is not a $7^+$-frame, there is no codeword in position {\bf Fe} and only one codeword among positions {\bf Dc} and {\bf Ec}.

Assume there is a codeword in position {\bf Ec} and no codeword in position {\bf Dc}. Due to Condition 2, there is a codeword in {\bf De}. Hence, by Observation~\ref{obs:NC}, $F+(-1,1)$ is a $6^+$-frame. If $F+(-1,0)$ is a $6^+$-frame, we are done. Otherwise, $F+(2,1)$ is a $6^+$-frame and we are also done.

Assume now that there is no codeword in position {\bf Ec}. Then by Condition 2, there is a codeword in position {\bf Dc}.

If there is a codeword in position {\bf De}, $F+(-1,0)$ is a $6^+$-frame and, by Observation~\ref{obs:CO}, $F+(-1,1)$ is a $6^+$-frame and we are done. Hence we may suppose that there is no codeword in position {\bf De}. This implies that there is a codeword in position {\bf Df}, and no codeword in position {\bf Ef}.

If there is a codeword in position {\bf Cf}, by Observation~\ref{obs:NC}, $F+(0,1)$ is a $6^+$-frame. Moreover, by Observation~\ref{obs:CO}, $F+(1,0)$ is a $6^+$-frame and we are done. So, we may assume there is no codeword in position {\bf Cf}, which implies that there is a codeword in position {\bf Ff}.

If there is a codeword in position {\bf Gd}, by Observation~\ref{obs:NC}, $F+(-1,-1)$ and $F+(1,-2)$ are $6^+$-frames and we are done. Hence we may assume that there is no codeword in position {\bf Gd}.

If there is a codeword in position {\bf Fg}, $F+(1,0)$ is a $6^+$-frame and by Observation~\ref{obs:NC}, $F+(2,0)$ is a $6^+$-frame and we are done. So we may assume that there is no codeword in position {\bf Fg}.

If there is a codeword in position {\bf Db}, $F+(-1,0)$ is a $6$-frame and by Observation~\ref{obs:NC}, $F+(-2,-1)$ is $6^+$-frame. If $F+(-1,-1)$ is a $6^+$-frame we are done, otherwise, by Observation~\ref{obs:NC}, $F+(0,-2)$ is a $6^+$-frame. Hence we may assume that there is no codeword in position {\bf Db}.

If there is no codeword in position {\bf Bf}, then there is a codeword in position {\bf Bc}, and by Observation~\ref{obs:NC}, $F+(-2,1)$ is a $6^+$-frame. Then, there must be a codeword among positions {\bf Bd} and {\bf Be}, implying that $F+(-1,1)$ is a $6^+$-frame. Hence we may assume that there is a codeword in position {\bf Bf}.

If there is a codeword in position {\bf Be}, by Observation~\ref{obs:NC}, $F+(1,1)$ and $F+(2,2)$ are $6^+$-frames and we are done. Hence we assume that there is no codeword in position {\bf Be}.

If there is no codeword in position {\bf Eg}, by Observation~\ref{obs:NC}, $F+(2,-1)$ and $F+(2,2)$ are $6^+$-frames. Now, if there is a codeword in {\bf Bd}, then by Observation~\ref{obs:NC}, $F+(1,2)$ is a $6^+$-frame and we are done. Otherwise, there is a codeword in {\bf Bc} and {\bf Ba}, and $F+(-2,1)$ is a $6$-frame, and we are done. So, assume that there is a codeword in position {\bf Eg}.

If there is a codeword in position {\bf Dg}, then $F+(1,0)$ is a $6^+$-frame and by Observation~\ref{obs:NC}, $F+(2,-1)$ is a $6^+$-frame. So we assume there is no codeword in position {\bf Dg}, which implies that there is a codeword in position {\bf Gg} by Condition 1.

If there is a codeword in position {\bf Gf}, by Observation~\ref{obs:NC}, $F+(0,-2)$ and $F+(2,-1)$ are $6^+$-frames. Now, if there is a codeword in position {\bf Ge}, $F+(1,-1)$ is a $6^+$-frame and we are done. Otherwise, there is a codeword in position {\bf Gb} and $F+(-2,-1)$ is a $6^+$-frame.
So we assume that there is no codeword in position {\bf Gf}, which implies that there is a codeword in position {\bf Ge}.

If there is a codeword in position {\bf Gb}, $F+(-1,-1)$ and $F+(-2,-1)$ are $6$-frames and we are done. So we assume that there is no codeword in position {\bf Gb}.

If there is a codeword in position {\bf Bd}, by Observation~\ref{obs:NC}, $F+(1,2)$ is a $6^+$-frame. Moreover $F+(1,1)$ is a $6^+$-frame so we are done. 
So we assume there is no codeword in position {\bf Bd}. Then there is a codeword in both positions {\bf Ba} and {\bf Bc}, implying that $F+(-2,1)$ is a $6$-frame.

If there is a codeword in position {\bf Cg}, $F+(1,0)$ is a $6$-frame and we are done. So, we assume there is no codeword in {\bf Cg}. Then, by Observation~\ref{obs:NC}, $F+(2,1)$ is a $6$-frame.

Now, if $F+(1,-2)$ is a $6^+$-frame we are done. Otherwise, there is no codeword in position {\bf Hd} and exactly one codeword among positions {\bf He} and {\bf Hf}. Since $F+(0,-2)$ is a $5^+$-frame, there is a codeword in position {\bf Hc}. By Condition 1, there is a codeword in position {\bf Ha}. This implies that $F+(-2,-2)$ is a $7$-frame and finishes the case analysis.
\end{proof}

%
A frame is called {\em $1$-poor} (resp. {\em $2$-poor}) if it is a $5$-frame having at most one $6$-frame in its $2$-ball (resp. no $6^+$-frame at distance~$1$ and at most two $6$-frames at distance~$2$).

\begin{lemma}\label{lem:five}
Let $F$ be a $5$-frame, then $F$ has at least one $6^+$-frame in its $2$-ball.
Moreover, one of the following properties holds:
\begin{itemize}
\item $F$ has a $6$-frame at distance~$1$ and another $6^+$-frame in its $2$-ball,
\item $F$ has a total charge excess of at least $3$ within its $2$-ball, with at least two $6^+$-frames,
\item $F$ is $1$-poor and, up to symmetry, the configuration around $F$ is depicted on Figure \ref{fig:1poor},
\item $F$ is $2$-poor and, up to symmetry, the configuration around $F$ is depicted on Figure \ref{fig:2poor}.
\end{itemize}
\end{lemma}


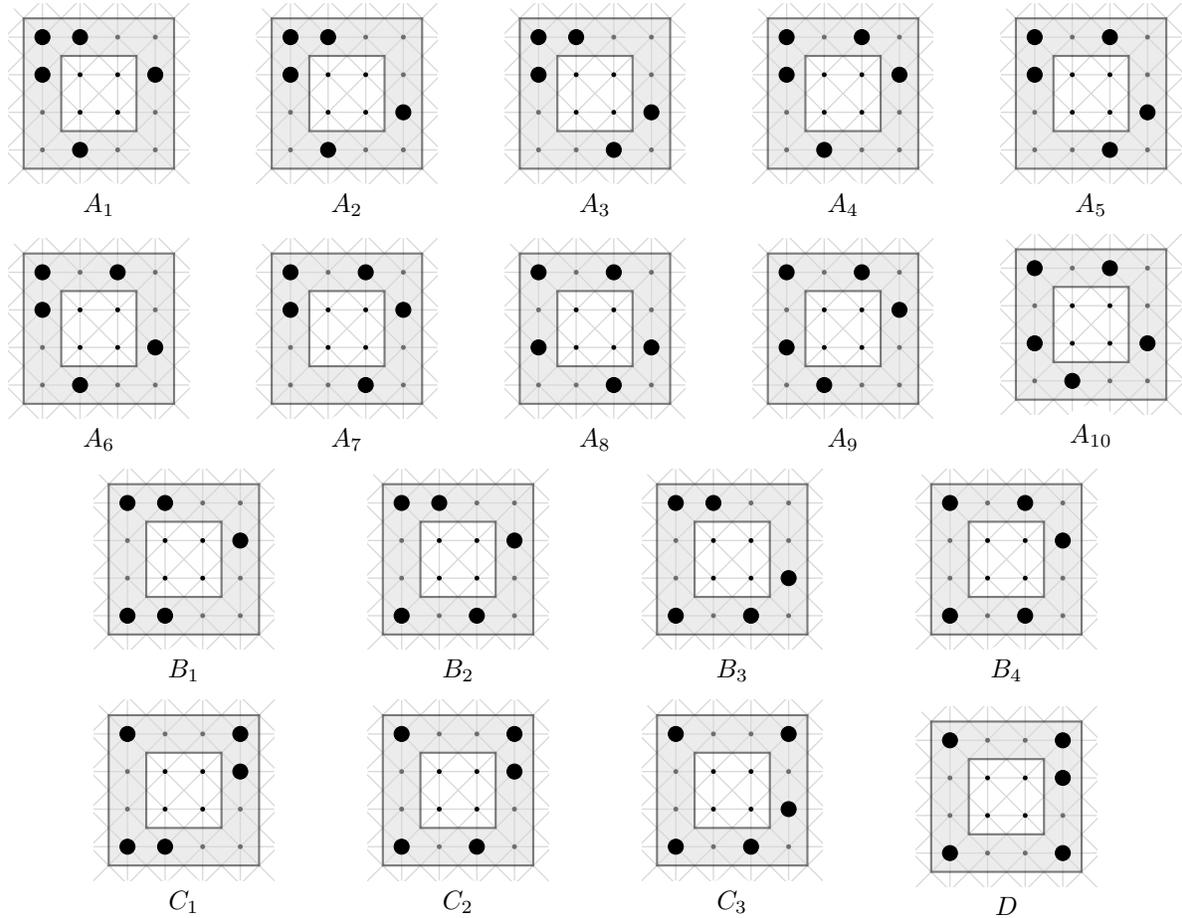
\begin{figure}[h]
\begin{center}

\begin{tikzpicture}[scale=0.5]
\framegrid

\begin{scope}[even odd rule]
\filldraw[patternus] (-2.5,-2.5) rectangle (1.5,1.5)
 (-1.5,-1.5) rectangle (0.5,0.5);
\end{scope}
\foreach \pos in {(-2,1),(-2,0),(-1,1),(-1,-2),(1,0)}
	{\node[code] at \pos {};}
\node[labelnode2] at (-0.5,-3.5) {$A_1$};
\end{tikzpicture}
\hfill
\begin{tikzpicture}[scale=0.5]
\framegrid

\begin{scope}[even odd rule]
\filldraw[patternus] (-2.5,-2.5) rectangle (1.5,1.5)
 (-1.5,-1.5) rectangle (0.5,0.5);
\end{scope}
\foreach \pos in {(-2,1),(-2,0),(-1,1),(-1,-2),(1,-1)}
	{\node[code] at \pos {};}
\node[labelnode2] at (-0.5,-3.5) {$A_2$};
\end{tikzpicture}
\hfill
\begin{tikzpicture}[scale=0.5]
\framegrid
\begin{scope}[even odd rule]
\filldraw[patternus] (-2.5,-2.5) rectangle (1.5,1.5)
 (-1.5,-1.5) rectangle (0.5,0.5);
\end{scope}
\foreach \pos in {(-2,1),(-2,0),(-1,1),(0,-2),(1,-1)}
	{\node[code] at \pos {};}
\node[labelnode2] at (-0.5,-3.5) {$A_3$};
\end{tikzpicture}
\hfill
\begin{tikzpicture}[scale=0.5]
\framegrid
\begin{scope}[even odd rule]
\filldraw[patternus] (-2.5,-2.5) rectangle (1.5,1.5)
 (-1.5,-1.5) rectangle (0.5,0.5);
\end{scope}
\foreach \pos in {(-2,1),(-2,0),(0,1),(-1,-2),(1,0)}
	{\node[code] at \pos {};}
\node[labelnode2] at (-0.5,-3.5) {$A_4$};
\end{tikzpicture}
\hfill
\begin{tikzpicture}[scale=0.5]
\framegrid

\begin{scope}[even odd rule]
\filldraw[patternus] (-2.5,-2.5) rectangle (1.5,1.5)
 (-1.5,-1.5) rectangle (0.5,0.5);
\end{scope}
\foreach \pos in {(-2,1),(-2,0),(0,1),(0,-2),(1,-1)}
	{\node[code] at \pos {};}
\node[labelnode2] at (-0.5,-3.5) {$A_5$};
\end{tikzpicture}

\begin{tikzpicture}[scale=0.5]
\framegrid

\begin{scope}[even odd rule]
\filldraw[patternus] (-2.5,-2.5) rectangle (1.5,1.5)
 (-1.5,-1.5) rectangle (0.5,0.5);
\end{scope}
\foreach \pos in {(-2,1),(-2,0),(0,1),(-1,-2),(1,-1)}
	{\node[code] at \pos {};}
\node[labelnode2] at (-0.5,-3.5) {$A_6$};
\end{tikzpicture}
\hfill
\begin{tikzpicture}[scale=0.5]
\framegrid

\begin{scope}[even odd rule]
\filldraw[patternus] (-2.5,-2.5) rectangle (1.5,1.5)
 (-1.5,-1.5) rectangle (0.5,0.5);
\end{scope}
\foreach \pos in {(-2,1),(-2,0),(0,1),(0,-2),(1,0)}
	{\node[code] at \pos {};}
\node[labelnode2] at (-0.5,-3.5) {$A_7$};
\end{tikzpicture}
\hfill
\begin{tikzpicture}[scale=0.5]
\framegrid
\begin{scope}[even odd rule]
\filldraw[patternus] (-2.5,-2.5) rectangle (1.5,1.5)
 (-1.5,-1.5) rectangle (0.5,0.5);
\end{scope}
\foreach \pos in {(-2,1),(-2,-1),(0,1),(0,-2),(1,-1)}
	{\node[code] at \pos {};}
\node[labelnode2] at (-0.5,-3.5) {$A_8$};
\end{tikzpicture}
\hfill
\begin{tikzpicture}[scale=0.5]
\framegrid

\begin{scope}[even odd rule]
\filldraw[patternus] (-2.5,-2.5) rectangle (1.5,1.5)
 (-1.5,-1.5) rectangle (0.5,0.5);
\end{scope}
\foreach \pos in {(-2,1),(-2,-1),(0,1),(-1,-2),(1,0)}
	{\node[code] at \pos {};}
\node[labelnode2] at (-0.5,-3.5) {$A_9$};
\end{tikzpicture}
\hfill
\begin{tikzpicture}[scale=0.5]
\framegrid
\begin{scope}[even odd rule]
\filldraw[patternus] (-2.5,-2.5) rectangle (1.5,1.5)
 (-1.5,-1.5) rectangle (0.5,0.5);
\end{scope}
\foreach \pos in {(-2,1),(-2,-1),(0,1),(-1,-2),(1,-1)}
	{\node[code] at \pos {};}
\node[labelnode2] at (-0.5,-3.5) {$A_{10}$};
\end{tikzpicture}

\begin{tikzpicture}[scale=0.5]
\framegrid

\begin{scope}[even odd rule]
\filldraw[patternus] (-2.5,-2.5) rectangle (1.5,1.5)
 (-1.5,-1.5) rectangle (0.5,0.5);
\end{scope}
\foreach \pos in {(-2,1),(-2,-2),(-1,1),(-1,-2),(1,0)}
	{\node[code] at \pos {};}
\node[labelnode2] at (-0.5,-3.5) {$B_1$};
\end{tikzpicture}
\hfil
\begin{tikzpicture}[scale=0.5]
\framegrid

\begin{scope}[even odd rule]
\filldraw[patternus] (-2.5,-2.5) rectangle (1.5,1.5)
 (-1.5,-1.5) rectangle (0.5,0.5);
\end{scope}
\foreach \pos in {(-2,1),(-2,-2),(-1,1),(0,-2),(1,0)}
	{\node[code] at \pos {};}
\node[labelnode2] at (-0.5,-3.5) {$B_2$};
\end{tikzpicture}
\hfil
\begin{tikzpicture}[scale=0.5]
\framegrid

\begin{scope}[even odd rule]
\filldraw[patternus] (-2.5,-2.5) rectangle (1.5,1.5)
 (-1.5,-1.5) rectangle (0.5,0.5);
\end{scope}
\foreach \pos in {(-2,1),(-2,-2),(-1,1),(0,-2),(1,-1)}
	{\node[code] at \pos {};}
\node[labelnode2] at (-0.5,-3.5) {$B_3$};
\end{tikzpicture}
\hfil
\begin{tikzpicture}[scale=0.5]
\framegrid
\begin{scope}[even odd rule]
\filldraw[patternus] (-2.5,-2.5) rectangle (1.5,1.5)
 (-1.5,-1.5) rectangle (0.5,0.5);
\end{scope}
\foreach \pos in {(-2,1),(-2,-2),(0,1),(0,-2),(1,0)}
	{\node[code] at \pos {};}
\node[labelnode2] at (-0.5,-3.5) {$B_4$};
\end{tikzpicture}

\begin{tikzpicture}[scale=0.5]
\framegrid

\begin{scope}[even odd rule]
\filldraw[patternus] (-2.5,-2.5) rectangle (1.5,1.5)
 (-1.5,-1.5) rectangle (0.5,0.5);
\end{scope}
\foreach \pos in {(-2,1),(-2,-2),(-1,-2),(1,1),(1,0)}
	{\node[code] at \pos {};}
\node[labelnode2] at (-0.5,-3.5) {$C_1$};
\end{tikzpicture}
\hfil
\begin{tikzpicture}[scale=0.5]
\framegrid

\begin{scope}[even odd rule]
\filldraw[patternus] (-2.5,-2.5) rectangle (1.5,1.5)
 (-1.5,-1.5) rectangle (0.5,0.5);
\end{scope}
\foreach \pos in {(-2,1),(-2,-2),(0,-2),(1,1),(1,0)}
	{\node[code] at \pos {};}
\node[labelnode2] at (-0.5,-3.5) {$C_2$};
\end{tikzpicture}
\hfil
\begin{tikzpicture}[scale=0.5]
\framegrid
\begin{scope}[even odd rule]
\filldraw[patternus] (-2.5,-2.5) rectangle (1.5,1.5)
 (-1.5,-1.5) rectangle (0.5,0.5);
\end{scope}
\foreach \pos in {(-2,1),(-2,-2),(0,-2),(1,1),(1,-1)}
	{\node[code] at \pos {};}
\node[labelnode2] at (-0.5,-3.5) {$C_3$};
\end{tikzpicture}
\hfil
\begin{tikzpicture}[scale=0.5]
\framegrid

\begin{scope}[even odd rule]
\filldraw[patternus] (-2.5,-2.5) rectangle (1.5,1.5)
 (-1.5,-1.5) rectangle (0.5,0.5);
\end{scope}
\foreach \pos in {(-2,1),(-2,-2),(1,-2),(1,1),(1,0)}
	{\node[code] at \pos {};}
\node[labelnode2] at (-0.5,-3.5) {$D$};
\end{tikzpicture}
\end{center}
\caption{\label{fig:5frames} All the possibilities for a $5$-frame.}
\end{figure}


\begin{proof}
Let $F$ be $5$-frame. We make a case analysis according to the configurations of $F$ depicted on Figure~\ref{fig:5frames}. Note that cases $A_1$ to $A_10$ represent
cases where $F$ has only one corner codeword, cases $B_1$ to $B_4$ are those
cases where $F$ has exactly two corner codewords on the same side,
cases $C_1$ to $C_3$ are the cases where $F$ has three corner codewords, and
case $D$ is the case where $F$ has four corner codewords. The cases
which are symmetric to those of Figure~\ref{fig:5frames} will
follow from the same arguments.  Moreover, note that if $F$ had
exactly two corner codewords on opposite corners, by Condition 2, each
side of $F$ would contain an additional codeword and $F$ would be a
$6^+$-frame.

\vspace{0.3cm}
{\bf Case $A_1$.} By Observation~\ref{obs:NC}, $F+(-1,0)$ and $F+(0,1)$ both are $6^+$-frames and we are done.

\vspace{0.3cm}
{\bf Case $A_2$.} By Observation~\ref{obs:NC}, $F+(-1,0)$ is a $6^+$-frame. By Observation~\ref{obs:CO}, $F+(0,1)$ is a $6^+$-frame and we are done.

\vspace{0.3cm}
{\bf Case $A_3$.} By Observation~\ref{obs:CO}, $F+(-1,0)$ and $F+(0,1)$ both are $6^+$-frames and we are done.

\vspace{0.3cm}
{\bf Case $A_4$.} By Observation~\ref{obs:NC}, $F+(0,1)$ is a $6^+$-frame. Now, by Condition 2, there is a codeword on either position {\bf Dd} or {\bf Ed}. In either case, by Observation~\ref{obs:NC}, $F+(-1,2)$ or $F+(1,-1)$ respectively, is a $6^+$-frame and we are done.

\vspace{0.3cm}
{\bf Case $A_5$.} By Observation~\ref{obs:CO}, $F+(0,1)$ is a $6^+$-frame. Now, by Condition 2, there is a codeword on either position {\bf Dd} or {\bf De}. In either case, by Observation~\ref{obs:NC}, $F+(-1,-1)$ or $F+(2,1)$ respectively, is a $6^+$-frame and we are done.

\vspace{0.3cm}
{\bf Case $A_6$.} By Observation~\ref{obs:CO}, $F+(0,1)$ is a $6^+$-frame. Now, by Condition 2, there is a codeword in either position {\bf Dd} or {\bf Ed}. In both cases, by Observation~\ref{obs:NC}, $F+(-1,-1)$ or $F+(1,-1)$ respectively, is a $6^+$-frame and we are done.

\vspace{0.3cm}
{\bf Case $A_7$.} By Observation~\ref{obs:NC}, $F+(0,1)$ is a $6^+$-frame. If $F+(-1,-1)$ is a $6^+$-frame, we are done. Otherwise, from the assumptions there are already four non-corner codewords in $F+(-1,-1)$, hence $F+(-1,-1)$ is a $5$-frame and, by Condition~1, all its further non-corner positions contain no codeword. In particular, there is no codeword in both positions {\bf Dd} and {\bf Gd}. By Condition 1, there is a codeword on at least one of the positions {\bf Dg} and {\bf Gg}. Hence, by Observation~\ref{obs:NC}, $F+(2,-1)$ is a $6^+$-frame and we are done.

\vspace{0.3cm}
{\bf Case $A_8$.} Assume there is no codeword in position {\bf Dd}. Then, $F+(-1,0)$ and $F+(2,1)$ are $6^+$-frames and we are done. So, assume there is a codeword in position {\bf Dd}.

If there is a codeword in position {\bf Ed}, by Observation~\ref{obs:NC}, $F+(-2,0)$ and $F+(1,0)$ are $6^+$-frames. Hence we may assume there is no codeword in position {\bf Ed}, and by symmetry, no codeword in position {\bf De} either.

If there is a codeword in position {\bf Ee}, by Observation~\ref{obs:NC}, $F+(1,1)$ and $F+(-1,-1)$ are $6^+$-frames. Hence we may assume there is no codeword in position {\bf Ee}.

If there is a codeword in position {\bf Be}, by Observation~\ref{obs:NC}, $F+(1,1)$ and $F+(2,2)$ are $6^+$-frames. Hence we may assume there is no codeword in position {\bf Be}, and by symmetry, no codeword in position {\bf Eb} either. By Condition 1, this implies that there is a codeword in position {\bf Bb}.

If there is a codeword in position {\bf Eg}, by Observation~\ref{obs:NC}, $F+(1,0)$ and $F+(2,1)$ are $6^+$-frames. Hence we may assume there is no codeword in position {\bf Eg}, and by symmetry, no codeword in position {\bf Ge} either.

If there is no codeword in position {\bf Bd}, by Observation~\ref{obs:NC}, $F+(0,2)$ and $F+(-2,1)$ are $6^+$-frames. If $F+(1,1)$ is a $6^+$-frame, we are done. Otherwise, by Observation~\ref{obs:NC}, $F+(2,1)$ is a $6^+$-frame and we are done too. Hence we may assume there is a codeword in position {\bf Bd}, and by symmetry, another codeword in position {\bf Db}.

If there is a codeword in position {\bf Bc}, $F+(0,1)$ is a $6$-frame and by Observation~\ref{obs:NC}, $F+(-1,1)$ is a $6^+$-frame, so we are done. Hence we may assume there is no codeword in position {\bf Bc}, and by symmetry, no codeword in position {\bf Cb} either.

If there is a codeword in position {\bf Gb}, $F+(-1,-1)$ and $F+(-1,-2)$ are $6^+$-frames. Hence we may assume there is no codeword in position {\bf Gb}, and by symmetry, no codeword in position {\bf Bg} either.

If there is no codeword in position {\bf Gf}, there is, by Condition 1, a codeword in position {\bf Gc} and then $F+(-1,-1)$ is a $6$-frame. Now, if there is no codeword in position {\bf Fg}, $F+(1,-1)$ is a $4$-frame and by Lemma~\ref{lem:four}, $F+(1,1)$ is a $6^+$-frame and we are done. Otherwise, by Observation~\ref{obs:NC}, $F+(1,-2)$ is a $6^+$-frame. Hence we may assume that there is a codeword in position {\bf Gf}, and by symmetry, another codeword in position {\bf Fg}.

If there is a codeword in position {\bf Gg}, $F+(1,-1)$ is a $6$-frame and by Observation~\ref{obs:NC}, $F+(1,-2)$ is a $6^+$-frame, so we are done. Hence we may assume there is no codeword in position {\bf Gg}.

If there is a codeword in position {\bf Gc}, $F+(0,-1)$ and $F+(-1,-1)$ is a $6^+$-frame, so we are done. Hence we may assume there is no codeword in positions {\bf Gc} and {\bf Cg}.

If $F+(-1,2)$ is a $5$-frame, by Condition 2, there is no codeword in position {\bf Ae} and a codeword on exactly one of the positions {\bf Ac} and {\bf Ad}. Then there must be a codeword in position {\bf Af} because $F+(0,2)$ has a fifth codeword. But then $F+(2,2)$ is a $6^+$-frame. Hence, there is a $6^+$-frame among $F+(-1,2)$ and $F+(2,2)$. By symmetry, there is also a $6^+$-frame among $F+(-2,1)$ and $F+(-2,-2)$.

If $F+(-1,-2)$ is a $6^+$-frame, we are done. Otherwise, there is no codeword in position {\bf Hd}, which implies that there is a codeword in position {\bf Hg}, as well as one codeword among positions {\bf He} and {\bf Hf}. By symmetry, if $F+(2,1)$ is a $5$-frame, there a codeword in position {\bf Gh}, as well as a codeword among positions {\bf Eh} and {\bf Fh}. Then, $F+(2,-2)$ is a $6^+$-frame and we are done.

\vspace{0.3cm}
{\bf Case $A_9$.} By Condition 2, there are at least two codewords among positions {\bf Dd}, {\bf De}, {\bf Ed} and {\bf Ee}. If there are two codewords on the same column or line, (resp. {\bf d}, {\bf D}, {\bf e}, {\bf E}), then by Observation~\ref{obs:NC}, resp. $F+(1,0)$, $F+(0,-1)$, $F+(-1,0)$, $F+(0,1)$, is a $6^+$-frame. Hence, if there are three codewords among positions {\bf Dd}, {\bf De}, {\bf Ed} and {\bf Ee}, there are two $6^+$-frames at distance~1 of $F$ and we are done.

Therefore, we assume that there are exactly two codewords among positions {\bf Dd}, {\bf De}, {\bf Ed} and {\bf Ee}. Then by Condition 2, they cannot lie on a same line or column. If there are codewords in positions {\bf De} and {\bf Ed}, by Observation~\ref{obs:NC}, $F+(1,-1)$ and $F+(-1,1)$ are $6^+$-frames and we are done. Hence, we assume that there are codewords in positions {\bf Dd} and {\bf Ee}.

If there is a codeword in position {\bf Eb}, by Observation~\ref{obs:NC}, $F+(-2,1)$ is a $6^+$-frame, and by Condition 2, $F+(-1,1)$ is a $6^+$-frame as well. Hence we may assume there is no codeword in position {\bf Eb}.

If there is a codeword in position {\bf Cb}, by Observation~\ref{obs:NC}, $F+(-2,0)$ is a $6^+$-frame, and by Condition 2, $F+(-1,0)$ is a $6^+$-frame as well. Hence we may assume there is no codeword in position {\bf Cb}.

If there is a codeword in position {\bf Gd}, $F+(1,-1)$ is necessarily a $6^+$-frame, and by Observation~\ref{obs:NC}, $F+(1,-2)$ is a $6^+$-frame. Hence we may assume there is no codeword in position {\bf Gd}.

By symmetry, we can assume that there is no codeword in any of the positions {\bf Be}, {\bf Bc} and {\bf Dg}. Using Condition 2, this implies that there are codewords in positions {\bf Db} and {\bf Bd}.

If there is a codeword in position {\bf Gg}, $F+(1,-1)$ is necessarily a $6^+$-frame, and by Observation~\ref{obs:NC}, $F+(2,-1)$ is a $6^+$-frame. Hence we may assume there is no codeword in position {\bf Gg}.

If there is a codeword in position {\bf Ge}, $F+(1,-1)$ is necessarily a $6^+$-frame. Now, if there is a codeword in position {\bf Fb}, $F+(-1,-1)$ is necessarily a $6$-frame too. Otherwise, either $F+(-2,-1)$, $F+(-1,1)$ or $F+(-2,2)$ is a $6^+$-frame, and we are done. Hence we may assume there is no codeword in position {\bf Ge} and by symmetry, there is no codeword in position {\bf Eg} either.

If there is a codeword in position {\bf Fb} (resp. in position {\bf Bf}), $F+(-1,0)$ (resp. $F+(0,1)$) is necessarily a $6$-frame. If there is a codeword in both positions, we are done. Hence by symmetry, without loss of generality, we can assume that there is no codeword in position {\bf Fb}. This implies that $F+(-2,-1)$ is a $6^+$-frame. If there is a $6^+$-frame at distance~1 of $F$, we are done. Otherwise, $F+(1,2)$ is a $6^+$-frame. If $F+(-2,2)$ is a $6^+$-frame, we are done. Otherwise, $F+(-2,-1)$ and $F+(1,2)$ are both necessarily $7^+$-frames and we are done too.

\vspace{0.3cm}
{\bf Case $A_{10}$.} Assume that there are no codewords among positions {\bf Ed} and {\bf Ee}. Then, by Condition 2, there are codewords on both positions {\bf Dd} and {\bf De}, and by Observation~\ref{obs:NC}, $F+(0,-1)$ and $F+(2,1)$ are $6^+$-frames.

Now, assume that there is a codeword in position {\bf Ed} or {\bf Ee}. Then, by Condition 2, $F+(0,1)$ is a $6^+$-frame. If there is a codeword in position {\bf Ed}, by Observation~\ref{obs:NC}, $F+(1,-1)$ is a $6^+$-frame. If there is a codeword in position {\bf Ee}, by Observation~\ref{obs:NC}, $F+(1,1)$ is a $6^+$-frame. In both cases we are done.

\vspace{0.3cm}
{\bf Case $B_1$.} By Observation~\ref{obs:NC}, $F+(-1,0)$ is a $6^+$-frame. By Condition 1, there is a codeword among positions {\bf Cb} and {\bf Fb}. Then, by Observation~\ref{obs:NC}, $F+(-2,0)$ is a $6^+$-frame and we are done.

\vspace{0.3cm}
{\bf Case $B_2$.} By Observation~\ref{obs:NC}, $F+(-1,0)$ is a $6^+$-frame. By Condition 2, there is a codeword among positions {\bf Dd} and {\bf Ed}. If there is a codeword in position {\bf Dd} (resp. {\bf Ed} and not in {\bf Dd}), by Observation~\ref{obs:NC}, $F+(1,1)$ (resp. $F+(1,-1)$) is a $6^+$-frame and we are done.

\vspace{0.3cm}
{\bf Case $B_3$.} By Observation~\ref{obs:NC}, $F+(-1,0)$ is a $6^+$-frame. If $F+(1,1)$ is a $6^+$-frame, we are done. Otherwise, $F+(1,1)$ is a $5$-frame and it has a unique corner codeword in position {\bf Ed}. By Condition 2, this implies that there is a codeword in position {\bf Bc}, and then by Observation~\ref{obs:NC}, $F+(0,2)$ is a $6^+$-frame and we are done.

\vspace{0.3cm}
{\bf Case $B_4$.} Assume first that there are no codewords among positions {\bf De} and {\bf Ee}. Then by Condition 2, there must a codeword on both positions {\bf Dd} and {\bf Ed}. Then by Observation~\ref{obs:NC}, $F+(1,0)$ and $F+(-2,0)$ are $6^+$-frames.

If there is a codeword in position {\bf De} or {\bf Ee}, then $F+(-1,0)$ is a $6^+$-frame. If there is a codeword in position {\bf De} (resp. {\bf Ee}), by Observation~\ref{obs:NC}, $F+(1,-1)$ (resp. $F+(2,-1)$) is a $6^+$-frame and we are done.

\vspace{0.3cm}
{\bf Case $C_1$.} By Observation~\ref{obs:NC}, both $F+(-1,0)$ and $F+(0,1)$ are $6^+$-frames.

\vspace{0.3cm}
{\bf Case $C_2$.} By Observation~\ref{obs:NC}, $F+(0,1)$ is a $6^+$-frame. By Condition 1, there is a codeword among positions {\bf Bc} and {\bf Bf}. Then, by Observation~\ref{obs:NC}, $F+(0,2)$ is a $6^+$-frame.

\vspace{0.3cm}
{\bf Case $C_3$.} Assume first that there is no codeword in position {\bf De}. Then, there must be codewords in both positions {\bf Dd} and {\bf Ee}. Then, by Observation~\ref{obs:NC}, $F+(1,1)$ and $F+(-1,-1)$ are both $6^+$-frames. Hence we can assume that there is a codeword in position {\bf De}, and, by symmetry in position {\bf Ed}.

If there is a codeword in position {\bf Dd}, by Observation~\ref{obs:NC}, $F+(0,2)$ and $F+(0,-1)$ are $6^+$-frames and we are done. Hence we may assume there is no codeword in position {\bf Dd}.

If there is a codeword in position {\bf Ee}, by Observation~\ref{obs:NC}, $F+(0,1)$ and $F+(2,0)$ are $6^+$-frames and we are done. Hence we may assume there is no codeword in position {\bf Ee}.

If there is a codeword in position {\bf Bc} or {\bf Bf}, by Observation~\ref{obs:NC}, $F+(0,2)$ is a $6^+$-frame and $F+(0,1)$ must be a $6^+$-frame too, so we are done. Hence we may assume there is no codeword in positions {\bf Bc} and {\bf Bf} and, by symmetry, on  positions {\bf Cb} and {\bf Fb}.

If there is a codeword in position {\bf Be}, then $F+(0,1)$ is a $6^+$-frame. Now, if either $F+(-1,1)$ or $F+(-2,2)$ is a $6^+$-frame, we are done. Otherwise, there is no codeword in position {\bf Ca} and hence there is a codeword in position {\bf Fa} by Condition 1.This implies that $F+(-2,-1)$ is a  $6^+$-frame, so we are done. Hence we may assume there is no codeword in position {\bf Be} and, by symmetry, in position {\bf Eb}.

If there is a codeword in position {\bf Eg}, by Observation~\ref{obs:NC}, $F+(2,1)$ and $F+(2,-2)$ are $6^+$-frames. If $F+(1,0)$ is a $6^+$-frame, we are done. Otherwise, there is no codeword in position {\bf Dg}, and a codeword among positions {\bf Cg}  and {\bf Fg}. If there is a codeword in position {\bf Fg}, by Observation~\ref{obs:NC}, $F+(1,-1)$ is a $6^+$-frame. If there is a codeword in position {\bf Cg}, by Observation~\ref{obs:NC}, $F+(2,0)$ is a $6^+$-frame. Hence we may assume there is no codeword in position {\bf Eg} and, by symmetry, in position {\bf Ge}.

If there is a codeword in position {\bf Gc}, by Observation~\ref{obs:NC}, $F+(0,-2)$ and $F+(-2,-1)$ are $6^+$-frames. If there are only two $6$-frames and no additional $6^+$-frame in the $2$-ball of $F$, $F$ is $2$-poor, then, using the same techniques than previously, one can check that the codewords around $F$ are fixed as in Figure \ref{fig:2poor}  and we are done. Otherwise, there is a total charge excess of at least $3$ within the $2$-ball of $F$ and we are done too. Hence we may assume there is no codeword in position {\bf Gc} and, by symmetry, in position {\bf Cg}.

We note that $F+(1,-1)$ is a $6^+$-frame. If $F+(1,-1)$ is a $6$-frame and if it is the only $6$-frame within the $2$-ball of $F$, then, one can check that $F$ is $1$-poor and that the codewords around $F$ are fixed as in Figure \ref{fig:1poor}, and so, we are done. Otherwise, either $F+(1,-1)$ is a $7^+$-frame and then $F+(1,-2)$ is a $6^+$-frame, or $F+(1,-1)$ is a $6$-frame and there exists another $6^+$-frame within the $2$-ball of $F$. In both cases we are done.

\vspace{0.3cm}
{\bf Case $D$.} By Observation~\ref{obs:NC}, $F+(0,1)$ is a $6^+$-frame. By Condition 1, there is a codeword among positions {\bf Bc} and {\bf Bf}. Then, by Observation~\ref{obs:NC}, $F+(0,2)$ is a $6^+$-frame.\end{proof}

\begin{figure}[h]
\begin{center}
\begin{tikzpicture}[scale=0.4]
\begin{scope}
\clip (-4.7,-5.7) rectangle (4.7,3.7);
\draw[rotate = 45,mygrid] (-8.9,-8.9) grid [xstep=1.414,ystep =1.414] (6.9,6.9);
\draw[shift = {(1,0)}, rotate = 45, mygrid] (-7.9,-6.9) grid [xstep=1.414,ystep =1.414] (6.9,6.9);
\draw[mygrid] (-4.9,-5.9) grid (4.9,3.9);
\foreach \I in {-4,...,3}\foreach \J in {-4,...,3}
	{\node[gridnode](\I\J) at (\I,\J) {};}
\end{scope}

\foreach \I / \A in {3/A, 2/B, 1/C, 0/D, -1/E, -2/F, -3/G, -4/H}
	{\node[labelnode] at (-5.2,\I) {\A};}
\foreach \I / \A in {3/a, 2/b, 1/c, 0/d, -1/e, -2/f, -3/g, -4/h}
	{\node[labelnode2] at (-1-\I,4.2) {\A};}

\begin{scope}[even odd rule, shift ={(1,-1)}]
\def\mypath{(-2.5,-2.5) rectangle (1.5,1.5)  (-1.5,-1.5) rectangle (0.5,0.5)}
\draw[black] \mypath;
\pattern[pattern color=roug!20, pattern=north east lines] \mypath;
\end{scope}

\begin{scope}[even odd rule]
\filldraw[patternus] (-2.5,-2.5) rectangle (1.5,1.5)
 (-1.5,-1.5) rectangle (0.5,0.5);
\end{scope}

\begin{scope}[shift={(-1,-3)}]
\draw[draw=black!40] (-2.5,-2.5) rectangle (1.5,1.5)
 (-1.5,-1.5) rectangle (0.5,0.5);
\end{scope}
\begin{scope}[shift={(3,1)}]
\draw[draw=black!40] (-2.5,-2.5) rectangle (1.5,1.5)
 (-1.5,-1.5) rectangle (0.5,0.5);
\end{scope}

\begin{scope}[even odd rule]
\filldraw[patternus] (-2.5,-2.5) rectangle (1.5,1.5)
 (-1.5,-1.5) rectangle (0.5,0.5);
\end{scope}

\foreach \pos in {(-4,-4),(-4,-1),(-4,1),(-4,3), (-3,-4), (-3,-3),(-3,0),(-3,2),(-2,-2),(-2,1),(-2,3),(-1,-3),(-1,-1),(-1,2),(0,-4),(0,-2),(0,0),(0,3),(1,-3),(1,-1),(1,1),(2,-2),(2,0),(2,2),(3,-4),(3,-1),(3,2),(3,3)}
\node[code] at \pos {};

\begin{scope}[shift={(-1,-0.5)}]
\draw[patternus] (-6,-7) rectangle (-5,-6);
\node[anchor=mid] at (-1.8,-6.5) {$1$-poor frame $\bf F$};
\end{scope}
\begin{scope}[shift={(2.5,-1.8)}]
\draw[pattern color=roug!20, pattern=north east lines] (-6,-7) rectangle (-5,-6);

\node[anchor=mid] at (-1.5,-6.5) {$1$-poor benefactor};
\end{scope}

\begin{scope}[shift={(9,-7)}]
\node at (4,0) {$X$: not a $1$-poor frame};
\node at ( 4,-1) {$?$ : maybe a $1$-poor frame};
\end{scope}

\begin{scope}[shift={(7,-0.5)}]
\draw[draw=black!40] (-6,-7) rectangle (-5,-6);
\node[anchor=mid] at (-3,-6.5) {$6^+$-frame};
\end{scope}

\begin{scope}[shift={(12,-1)}, scale =1.2]
\draw[mygrid] (-3,-3) grid (4,4);
\draw[patternus] (0,0) rectangle (1,1);
\draw[line width=2pt] (-2,-2) rectangle (3,3);
\draw[pattern color=roug!20, pattern=north east lines] (1,-1) rectangle (2,0);
\node at (0.5,0.5) {\Large $\bf F$};
\node at (1.5,-0.5) {$6$};
\node at (3.5,1.5) {$6^+$};
\node at (-0.5,-2.5) {$6^+$};
\node at (0.5,-1.5) {$X$};
\node at (2.5,0.5) {$X$};
\node at (2.5,-1.5) {$?$};
\end{scope}
\end{tikzpicture}
\end{center}
\caption{\label{fig:1poor} Neighbourhood of a $1$-poor frame in the vertex lattice (on the left) and in the frame lattice (on the right).}
\end{figure}
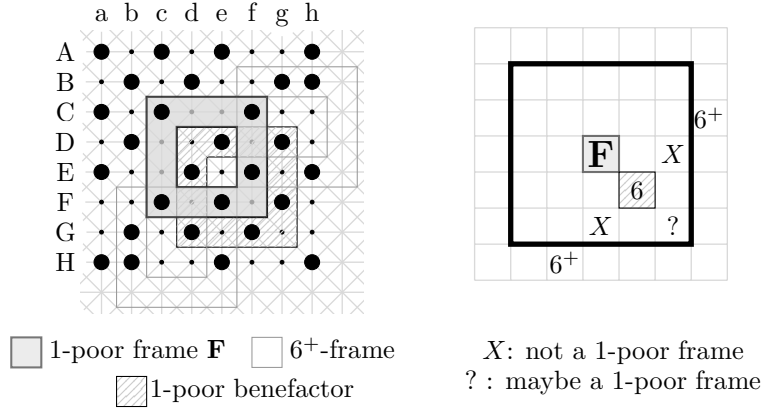

\begin{lemma}\label{lem:1poor}
If $F$ is a $1$-poor frame, then $F$ has a $6$-frame  $F_6$ at distance~$1$ such that:
\begin{itemize}
\item $F_6$ has at most two $1$-poor frames in its $2$-ball and each of them is at distance~$1$ of $F_6$,
\item $F_6$ has at least two  $6^+$-frames at distance~$2$,
\item $F_6$ is not a $4$-benefactor.
\end{itemize}
\end{lemma}

\begin{proof}
By Lemma~\ref{lem:five}, if $F$ is a $1$-poor frame, the neighbourhood of $F$ is fixed and depicted in Figure \ref{fig:1poor} (up to symmetry). The only $6^+$-frame in the neighbourhood of $F$ is $F_6=F+(1,-1)$ and contains exactly six codewords. The first point of the lemma follows from the fact that a $1$-poor frame must have a $6$-frame in a corner of its $1$-ball, however by using the left part of Figure \ref{fig:1poor} the positions marked by "X" on the right part of the figure cannot host a $1$-poor frame. The second point follows using Observation \ref{obs:NC} on the left part of Figure \ref{fig:1poor}. The last point follows by noting that a $4$-benefactor must have two non-codeword corners in the same column or line, which is not the case here.
\end{proof}

A frame playing the role of $F_6$ in the previous lemma is called a {\em $1$-poor-benefactor}.

\begin{figure}[h]
\begin{center}
\begin{tikzpicture}[scale=0.4]
\begin{scope}
\clip (-5.7,-5.7) rectangle (4.7,3.7);
\draw[rotate = 45,mygrid] (-8.9,-8.9) grid [xstep=1.414,ystep =1.414] (8.9,8.9);
\draw[shift = {(1,0)}, rotate = 45, mygrid] (-8.9,-8.9) grid [xstep=1.414,ystep =1.414] (8.9,8.9);
\draw[mygrid] (-5.9,-5.9) grid (4.9,3.9);
\foreach \I in {-4,...,3}\foreach \J in {-4,...,3}
	{\node[gridnode](\I\J) at (\I,\J) {};}
\end{scope}

\foreach \I / \A in {3/A, 2/B, 1/C, 0/D, -1/E, -2/F, -3/G, -4/H}
	{\node[labelnode] at (-6.2,\I) {\A};}
\foreach \I / \A in {3/a, 2/b, 1/c, 0/d, -1/e, -2/f, -3/g, -4/h}
	{\node[labelnode2] at (-1-\I,4.2) {\A};}
	
\begin{scope}[shift={(0,-2)}]
\draw[draw=black!40] (-2.5,-2.5) rectangle (1.5,1.5);
\end{scope}
\begin{scope}[shift={(-3,-1)}]
\draw[draw=black!40] (-2.5,-2.5) rectangle (1.5,1.5);
\end{scope}
\begin{scope}[shift={(-3,-2)}]
\draw[draw=black!40] (-2.5,-2.5) rectangle (1.5,1.5);
\end{scope}
\begin{scope}[shift={(-1,-3)}]
\draw[draw=black!40] (-2.5,-2.5) rectangle (1.5,1.5);
\end{scope}
\begin{scope}[even odd rule, shift ={(-2,-1)}]
\def\mypath{(-2.5,-2.5) rectangle (1.5,1.5)  (-1.5,-1.5) rectangle (0.5,0.5)}
\draw[black] \mypath;
\pattern[pattern color=roug!20, pattern=north east lines] \mypath;
\end{scope}

\begin{scope}[even odd rule]
\filldraw[patternus] (-2.5,-2.5) rectangle (1.5,1.5)
 (-1.5,-1.5) rectangle (0.5,0.5);
\end{scope}

\foreach \pos in {(-4,-3),(-4,-1),(-4,1),(-4,3),(-3,-4),(-3,-3),(-3,0),(-3,2),(-2,-3),(-2,-2),(-2,1),(-2,3),(-1,-4),(-1,-1),(-1,2),(0,-2),(0,0),(0,3),(1,-3),(1,-1),(1,1),(2,-4),(2,-2),(2,0),(2,2),(3,-4),(3,-1),(3,2),(3,3)}
\node[code] at \pos {};

\begin{scope}[shift={(-4.5,-0.5)}]
\draw[patternus] (-6,-7) rectangle (-5,-6);
\node[anchor=mid] at (-1.9,-6.5) {$2$-poor frame $\bf F$};
\end{scope}
\begin{scope}[shift={(8.2,-0.5)}]
\draw[pattern color=roug!20, pattern=north east lines] (-6,-7) rectangle (-5,-6);
\node[anchor=mid] at (-1.3,-6.5) {$2$-poor-benefactor};
\end{scope}
\begin{scope}[shift={(3.2,-0.5)}]
\draw[draw=black!40] (-6,-7) rectangle (-5,-6);
\node[anchor=mid] at (-3,-6.5) {$6^+$-frame};
\end{scope}
\begin{scope}[shift={(19.2,-0.5)}]
\node[anchor=mid] at (-4,-6.5) { $X$: not a $2$-poor frame};
\end{scope}

\begin{scope}[shift={(12,-1)}, scale =1.2]
\draw[mygrid] (-3,-3) grid (4,4);
\draw[patternus] (0,0) rectangle (1,1);
\draw[line width=2pt] (-2,-2) rectangle (3,3);
\draw[pattern color=roug!20, pattern=north east lines] (-2,-1) rectangle (-1,0);
\node at (0.5,0.5) {\Large $\bf F$};
\node at (-1.5,-0.5) {$6$};
\node at (0.5,-1.5) {$6$};
\node at (-2.5,-0.5) {$6^+$};
\node at (-2.5,-1.5) {$6^+$};
\node at (-0.5,-2.5) {$6^+$};
\node at (-0.5,1.5) {$X$};
\node at (-2.5,1.5) {$X$};
\end{scope}
\end{tikzpicture}
\end{center}
\caption{\label{fig:2poor} Neighbourhood of  a $2$-poor frame in the vertex lattice (on the left) and in the frame lattice (on the right).}
\end{figure}
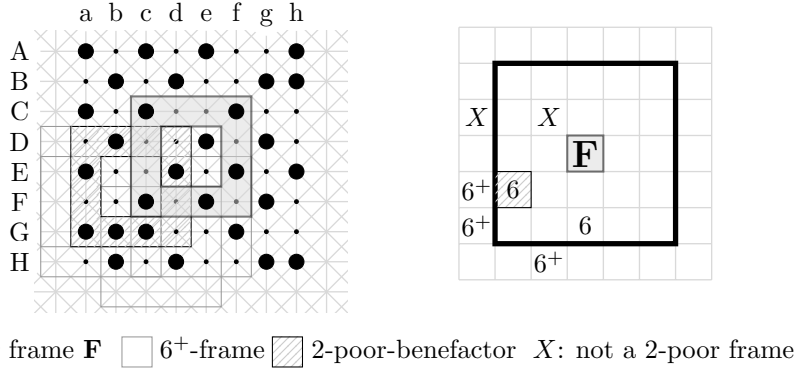

\begin{lemma}\label{lem:2poor}
If $F$ is a $2$-poor frame, then $F$ has two $6$-frames at distance~$2$. One of them, say $F_6$, has the following properties:
\begin{itemize}
\item $F_6$ is neither a $4$-benefactor nor a $1$-poor-benefactor,
\item $F_6$ has at least four $6^+$ frames in its $2$-ball,
\item $F_6$ has only one $2$-poor frame in its $2$-ball.
\end{itemize}
\end{lemma}
\begin{proof}
By Lemma~\ref{lem:five}, if $F$ is a $2$-poor frame, the neighbourhood of $F$ is fixed and depicted in Figure \ref{fig:2poor} (up to symmetry). Let $F_6=F+(-2,-1)$. The first point of the lemma follows from the fact that, by Lemma \ref{lem:five}, a $1$-poor benefactor must have a $6$-frame in its $2$-ball (and only this $6$-frame in its $2$-ball) and that a $4$-benefactor must have two non-codeword corners in the same line or column (hence, the only possibility would be that $F+(-2,1)$ is a $4$-frame which is not true). The second point follows by using Observation \ref{obs:NC} on the left part of Figure \ref{fig:2poor}. For the third point, observe that the only possible locations for another $2$-poor frame in the $2$-ball of $F_6$ are those marked by an "X" on the right part of Figure \ref{fig:2poor}: $F+(-3,1)$ and $F+(-1,1)$. However, both have two $6^+$-frames in their $2$-ball that do not match the codeword configuration of a $2$-poor frame as given on the left part of Figure \ref{fig:2poor}.
\end{proof}

A frame playing the role of $F_6$ in the previous lemma is called a {\em $2$-poor-benefactor}.

\section{The discharging procedure}

We are now ready to describe the discharging procedure which leads to our result.
In the beginning, each $k$-frame has charge $k$.
Let $\alpha=\tfrac{1}{37}$. We apply the following rules:

\begin{enumerate}
\item Let $F$ be a $4$-benefactor $6$-frame without any $7^+$-frame among its co-benefactors.
Assume $F$ is oriented as in Figure~\ref{fig:4benef} (other cases follow by symmetry). 
We consider two subrules:
\begin{enumerate}
\item If there is a $6^+$-frame in a $\Y$-position of Figure~\ref{fig:4benef}, say in position $F+(-1,1)$ (the other case will be covered by symmetry), and no other $6^+$-frames in $\Y$- and $\Z$-positions of Figure~\ref{fig:4benef}, then $F$ gives charge $\tfrac{3\alpha+1}{4}$ to the $4$-frame  $F+(-2,0)$, charge $\alpha$ to frame $F+(0,1)$, charge $2\alpha $ to the other $5$-frames in $\Z$-positions of Figure~\ref{fig:4benef}, charge $\tfrac{3\alpha}{2}$ to the other  $5$-frame in $\Y$-position of Figure~\ref{fig:4benef} and charge $\alpha$ to all the other $5$-frames in the $2$-ball of $F$.

\item Otherwise, $F$ gives charge $\tfrac{3\alpha+1}{4}$ to the $4$-frame  $F+(-2,0)$, charge $2\alpha $ to $5$-frames in $\Z$-positions of Figure~\ref{fig:4benef}, charge $\tfrac{3\alpha}{2}$ to $5$-frames in $\Y$-positions of Figure~\ref{fig:4benef} and charge $\alpha $ to all the other $5$-frames in the $2$-ball of $F$.
\end{enumerate}

\item Let $F$ be a $4$-benefactor $6$-frame with a $7^+$-frame among its co-benefactors. Assume $F$ is oriented as in Figure~\ref{fig:4benef}, and that $F+(-2,2)$ is a $7^+$-frame  (other cases follow by symmetry).
Then $F$ gives the same charges than in Rule 1b, except for the frame $F+(-1,1)$ which receives no charge from $F$.

\item A $4$-benefactor $7^+$-frame $F$ gives charge $\tfrac{3\alpha+1}{4}$ to the $4$-frames among  $F+\{(0,2), (0,-2), (2,0), (-2,0)\}$, charge $3\alpha $ to frames at distance~$1$ and charge $2\alpha$ to frames at distance~$2$.

\item A $1$-poor-benefactor, recall that it is a $6$-frame, gives charge $3\alpha $ to the $1$-poor frames in its $2$-ball, $2\alpha $ to the other $5$-frames at distance~1 and charge $\alpha $ to the other $5$-frames at distance~2.
\item A $2$-poor-benefactor, recall that it is a $6$-frame, gives charge $2\alpha $ to the unique $2$-poor frame in its $2$-ball, $2\alpha $ to the other $5$-frames at distance~1 and charge $\alpha $ to the other $5$-frames at distance~2.
\item Other $6$-frames give charge $2\alpha $ to $5$-frames at distance~1 and charge $\alpha $ to $5$-frames at distance~2.
\item Other $7^+$-frames give charge $3\alpha $ to $5$-frames at distance~1 and charge $2\alpha $ to $5$-frames at distance~2.
\end{enumerate}

We note that the rules are not ambiguous. Indeed, by Lemmas~\ref{lem:1poor} and~\ref{lem:2poor}, a $6$-frame can be either a $4$-benefactor, a $1$-poor benefactor, a $2$-poor benefactor or not a benefactor at all, but never two at the same time.

\begin{lemma}\label{lem:d7}
After the application of the discharging rules, each $7^+$-frame has charge at least $5+3\alpha $.
\end{lemma}

\begin{proof}
Let $F$ be a  $7^+$-frame. It is sufficient to show that $F$ gives at most charge $2-3\alpha=\tfrac{71}{37}$.

If $F$ is a $4$-benefactor, then by Lemma~\ref{lem:4benef}, it has at most two $4$-frames among the frames $F+\{(0,2), (0,-2), (2,0), (-2,0)\}$. 

If $F$ has two $4$-frames among  $F+\{(0,2), (0,-2), (2,0), (-2,0)\}$ then it has at least three co-benefactors. Hence, $F$ has at most eleven $5$-frames at distance~$2$. By Rule 3, $F$ gives at most charge $2\cdot \tfrac{3\alpha+1}{4} + 8\cdot 3\alpha + 11\cdot 2\alpha = \tfrac{66}{37}$. 

If $F$ has only one $4$-frame among $F+\{(0,2), (0,-2), (2,0), (-2,0)\}$, then it has two co-benefactors and hence, at most thirteen $5$-frames at distance~$2$. By Rule 3, $F$ gives at most charge $\tfrac{3\alpha+1}{4} + 8\cdot 3\alpha + 13\cdot 2\alpha = \tfrac{60}{37}$.

Finally, if $F$ is not a $4$-benefactor, then by Rule 7, it gives at most charge $8\cdot 3\alpha + 16 \cdot 2\alpha = \tfrac{56}{37}$.
\end{proof}

\begin{lemma}\label{lem:d6}
After the application of the discharging rules, each $6$-frame has charge at least $5+3\alpha $.
\end{lemma}

\begin{proof}
Let $F$ be a  $6$-frame. It is sufficient to show that $F$ gives at most charge $1-3\alpha=\tfrac{34}{37}$.

Assume first that $F$ is a $4$-benefactor without any $7^+$-frame among its co-benefactors.
If there were no other $6^+$-frames than the co-benefactors of $F$ within its $2$-ball, $F$ would give, according to Rule 1, charge $\tfrac{3\alpha+1}{4}+5\cdot 2\alpha + 2\cdot \frac{3\alpha}{2}+14\alpha=\tfrac{37}{37}$.
However, by Lemma~\ref{lem:4benef6}, $F$ has at least two non-co-benefactor $6^+$-frames in its 2-ball. 
Moreover by Lemma \ref{lem:4benef6} as well, either one of those is on a $\Y$- or $\Z$-position, or there is a third non-co-benefactor $6^+$-frame in the $2$-ball of $F$. We distinguish two cases.
If there is an extra $6^+$-frame on a $\Y$-position and one on an $\X$-position (Rule 1a applies), then we save charge at least $\tfrac{3\alpha}{2}+\alpha$ from the extra frames and charge $\alpha$ from frame $F+(0,1)$.
Otherwise (Rule 1b applies), the total charge saved on the extra $6^+$-frames is at least $3\alpha$. Indeed, we save at least $2\cdot \tfrac{3\alpha}{2}$ if the two extra $6^+$-frames are both on $Y$-positions, $2\cdot 2\alpha$ if they are both on $Z$-positions, $2\alpha+\tfrac{3\alpha}{2}$ for a $Z$-position and a $Y$-position, $2\alpha+\alpha$ for a $Z$-position and an $X$-position, and $3\cdot\alpha$ if there are three extra $6^+$-frames.
In all cases, $F$ gives at most $\tfrac{34}{37}$ of charge.

%

If $F$ is a $4$-benefactor with a $7^+$-frame among its co-benefactors, by Lemma~\ref{lem:4benef6}, there are two $6^+$-frames in the $2$-ball of $F$ in addition to its two co-benefactors and they are at distance~$2$ from $F$. Hence, $F$ has at most eleven $5$-frames at distance~$2$. By Rule $2$, $F$ gives at most $\tfrac{3\alpha+1}{4}+5\cdot 2\alpha+\tfrac{3\alpha}{2}+12\alpha = \tfrac{33.5}{37}$.

If $F$ is a $1$-poor-benefactor, by Lemma~\ref{lem:1poor}, $F$ has at most two $1$-poor frames in its $2$-ball and each of them is at distance~1 of $F$. Moreover, there are at least two $6^+$-frames at distance~2 of $F$. We may assume that $F$ has two $1$-poor frames in its $2$-ball since by Rule $4$, $F$ would give away more charge in this case. Now, by Rule $4$, $F$ gives at most $2\cdot 3\alpha + 6 \cdot 2\alpha + 14 \cdot \alpha = \tfrac{32}{37}$.

If $F$ is a $2$-poor-benefactor, by Lemma~\ref{lem:2poor}, $F$ has only one $2$-poor frame and at least four $6^+$-frames in its $2$-ball. We may assume that these frames are at distance~2 of $F$ since by Rule $5$, $F$ would give away more charge in this case. By Rule $5$, $F$ gives at most $2\alpha + 8 \cdot 2\alpha + 11\cdot \alpha = \tfrac{29}{37}$.

Finally, if $F$ is not a benefactor, then by Rule 6, it gives at most charge $8\cdot 2\alpha + 16 \cdot \alpha = \tfrac{32}{37}$.
\end{proof}

\begin{lemma}\label{lem:d5}
After the application of the discharging rules, each $5$-frame has charge at least $5+3\alpha $.
\end{lemma}

\begin{proof}
It is enough to prove that each $5$-frame receives charge $3\alpha$.

We first note that by our discharging rules, each $6$-frame $F_6$ gives at  least charge $\alpha$ to each $5$-frame of its $2$-ball, except in Rule 2 where $F_6$ does not give anything to one $5$-frame $F'$ at distance~$1$ of $F_6$. However, $F'$ has a $4$-benefactor $7^+$-frame at distance~$1$, which, by Rule 3, gives charge $3\alpha$ to $F'$.

Let $F$ be a $5$-frame. By the previous paragraph, we can consider that $F$ receives at least charge $\alpha$ from each $6$-frame in its $2$-ball. By Rules $3$ and $7$, $F$ receives at least charge $2\alpha$ from each $7^+$-frame in its $2$-ball. Hence, if $F$ has a total charge excess of at least $3$ within its $2$-ball, with two $6^+$-frames, $F$ receives at least charge $3\alpha$ and we are done.
Otherwise, by Lemma~\ref{lem:five}, $F$ is either $1$-poor, $2$-poor or has a $6$-frame at distance~$1$ and another one in its $2$-ball.

If $F$ is $1$-poor, by Lemma~\ref{lem:1poor}, it has a $1$-poor-benefactor in its $2$-ball which, by Rule 4, gives charge $3\alpha$ to $F$.

If $F$ is $2$-poor, by Lemma~\ref{lem:2poor}, it has a $2$-poor-benefactor in its $2$-ball which, by Rule 5, gives charge $2\alpha$ to $F$. Moreover, by Lemma~\ref{lem:2poor}, $F$ has another $6^+$-frame in its $2$-ball which gives charge $\alpha$ to $F$.

Finally, suppose $F$ has a $6$-frame $F_6$ at distance~$1$ and another one, $F'_6$, in its $2$-ball.
If $F$ receives charge $2\alpha$ from $F_6$, we are done. Otherwise, by our discharging rules $F_6$ is necessarily a $4$-benefactor. Without loss of generality, we can assume that $F_6$ is oriented as in Figure~\ref{fig:4benef}. Then $F_6$ gives always charge $2\alpha$ to the $5$-frames in $\Z$-positions, except in Rule $1a$ where $F_6+(0,1)$ receives only charge $\alpha$, but in this case, $F_6+(0,1)$ has three $6^+$-frames in its $2$-ball, so $F\neq F_6+(0,1)$.
Frame $F_6+(-1,0)$ has three $6^+$-frames in its $2$-ball, so $F\neq F_6+(-1,0)$.
Thus, we can assume that $F$ is on a $\Y$-position. Hence, $F'_6$ is also a $4$-benefactor.
In this case, by Rules 1 and 2, both $F_6$ and $F'_6$ give charge $\tfrac{3\alpha}{2}$ to $F$, and we are done.  
\end{proof}

\begin{lemma}\label{lem:d4}
After the application of the discharging rules, each $4$-frame has charge at least $5+3\alpha $.
\end{lemma}

\begin{proof}
By Lemma~\ref{lem:four}, a $4$-frame $F$ has four $4$-benefactors. By Rules $1$, $2$ and $3$,  each $4$-benefactor gives charge $\tfrac{3\alpha+1}{4}$ to $F$. Hence, $F$ receives charge $1+3\alpha $ and ends with charge $5+3\alpha $.
\end{proof}

After the application of our discharging rules, by Lemmas~\ref{lem:d7},~\ref{lem:d6},~\ref{lem:d5} and~\ref{lem:d4}, each frame has charge at least $5+3\alpha=\frac{188}{37}$, therefore, the average number of codewords in each frame is at least $\frac{188}{37}$ (this is due to the fact that each frame gives charge to vertices at distance at most 2). There are twelve vertices in each frame, hence, by Proposition~\ref{prop:dens}, we obtain Theorem \ref{thm:main}.

\end{document}